\documentclass[12pt, twoside,leqno]{article}
\usepackage{amsmath,amsthm}
\usepackage{amssymb}

\newtheorem{thm}{Theorem}[section]
\newtheorem{lem}[thm]{Lemma}
\newtheorem{conj}[thm]{Conjecture}
\newtheorem{prop}[thm]{Proposition}

\def\({\left(}
\def\){\right)}
\def\[{\left[}
\def\]{\right]}
\def\<{\langle}
\def\>{\rangle}

\def\Card{\operatorname{Card}}

\begin{document}

\baselineskip=17pt

\title{On the local structure of the set of values of Euler's $\varphi$ function}

\author{Jean-Marc Deshouillers\\
Institut de Math\'{e}matiques de Bordeaux\\
Universit\'{e} de Bordeaux, CNRS, Bordeaux INP\\
33400 Talence, France\\
E-mail: jean-marc.deshouillers@math.u-bordeaux.fr
\and
Pramod Eyyunni\\
IIT Gandhinagar, Palaj\\
Gandhinagar 382355,  Gujarat, India\\
E-mail: pramodeyy@gmail.com
\and
Sanoli Gun\\
Institute of Mathematical Sciences, HBNI\\
C.I.T Campus, Taramani\\
Chennai 600113, Tamil Nadu, India\\
E-mail: sanoli@imsc.res.in}

\date{}

\maketitle

\renewcommand{\thefootnote}{}



\renewcommand{\thefootnote}{\arabic{footnote}}
\setcounter{footnote}{0}

\begin{abstract}
Assuming the validity of Dickson's conjecture, we show that the set 
$\mathcal{V}$ of values of the Euler's totient function $\varphi$ contains arbitrarily 
large arithmetic progressions with common difference 4. This leads to the question 
of proving unconditionally that this set $\mathcal{V}$ has a positive 
upper Banach density.
\end{abstract}

\section{Introduction}

In \cite{DEP1}, Mithun Kumar Das, Bhuwanesh Rao Patil and the second named author of this paper have 
studied the local behaviour of some sequences of integers associated to Euler's totient function $\varphi$. 
One of the tools they use is the (upper) Banach density, also called upper uniform density. For a sequence 
of integers $\mathcal{A}$, its Banach density is defined by
\begin{equation}\label{defdelta}
\delta^*(\mathcal{A}) = \lim_{H \rightarrow \infty} \limsup_{x \rightarrow \infty} \frac{1}{H} 
\operatorname{Card} (\mathcal{A}\cap (x, x+H]).
\end{equation}

They claim that the sequence $\mathcal{V}$ of values of the $\varphi$  function over the 
positive integers has Banach density $0$ and notice in \cite{DEP2} that their argument 
is not complete. The main aim of this paper is to show that, assuming a standard 
conjecture of Dickson, the Banach density of $\mathcal{V}$ is $1/4$.\\

Let us first recall Dickson's conjecture \cite{DIC} which is a predecessor of the Hardy-Littlewood 
conjecture on prime $k$-tuples and also Schinzel's Hypothesis H.
\begin{conj}[Dickson's conjecture]\label{dickson}
Let $s$ be a positive integer and  $F_1, F_2,\ldots, F_s$ be $s$ linear polynomials with integral 
coefficients and positive leading coefficient such that their product has no fixed prime 
divisor \footnote{We say that a prime number $p$ is a \emph{fixed prime divisor} of a polynomial 
$G$ if we have: $\forall t \in \mathbb{Z} : p | G(t)$.}. Then, there exist infinitely many positive 
integers $t$ such that $F_1(t), F_2(t),\ldots, F_s(t)$ are all primes.
\end{conj}

Our main result is the following;
\begin{thm}\label{main}
Assuming the validity of Dickson's conjecture, the set of values of the Euler's totient function 
$\varphi$ has upper Banach density $1/4$.
\end{thm}

The lower bound  $\delta^*(\mathcal{V}) \ge 1/4$ is a direct consequence of Theorem \ref{strongmain} which is stated below and proved in Section \ref{two}. The upper bound  $\delta^*(\mathcal{V}) \le 1/4$
is easily derived, in Section \ref{three}, from a result of Ford-Konyagin-Pomerance \cite{FKP}. \\

\begin{thm}\label{strongmain}
Assume the validity of Dickson's conjecture. For any 
positive integer $H$, there exist positive integers 
$M, m_1, \ldots, m_H$ such that
\begin{equation}\label{estrong}
\forall h \in [1, H]  \colon ~~~
\varphi(m_{h}) = M+ 4h.
\end{equation}
\end{thm}

\noindent \textbf{Remark 1.}\;
It would be interesting to prove unconditionally that the Banach density of 
$\mathcal{V}$ is positive.\\

\section{Proof of Theorem \ref{strongmain}}\label{two}

 The letters $p$ and $q$, with or without index or subscript, 
 always denote prime numbers.\\

\textbf{Step 1}. Germain primes in arithmetic progressions\\

Although our subsequent construction will use Germain primes 
\footnote{ A prime number $p$ is said to be a (Sophie) Germain prime if 
$2p+1 $ is also a prime number.},
it is possible to avoid them and also the Dickson conjecture at this stage, 
but their use will  make our life easier.
\begin{lem}\label{ginap}
Assume the validity of Dickson's conjecture. Let $a$ be a positive integer and $b$ be 
an integer such that both $b$ and $2b+1$ are coprime to~$a$. The arithmetic 
progression with difference $a$ and first term $b$ contains infinitely many Germain primes.
\end{lem}
\begin{proof}
We show that the polynomial $G$ defined by 
$$
G(t)=(at+b)(2at+2b+1)
$$ 
has no fixed prime divisor. \\
First, $2$ is not a fixed divisor: otherwise, $a = (a\times 1+b)-(a\times 0+b)$ is even and 
then $b$ is even, a contradiction.\\
Next, if $p>2$ does not divide $a$, then the polynomial $G$ modulo $p$ is of degree two:
 it has at most two roots and cannot have $p$ roots.\\
Finally, if $p>2$ divides $a$, then the polynomial $G$ modulo $p$ is the constant $b(2b+1)$, 
which is coprime to $p$, otherwise we would have either $\gcd(a, b)>1$ or $\gcd(a, 2b+1)>1$, 
a contradiction to our hypothesis.

\smallskip
By Dickson's conjecture, there exist infinitely many $t$ such that $at+b$ is prime, 
as well as $2at+2b+1 = 2(at+b)+1.$
\end{proof}

\textbf{Step 2}. Construction of a suitable family of primes.\\

The proof will make use of a family of Germain primes satisfying 
some congruences and size properties, the existence of which is 
asserted by the following proposition.

\begin{prop}\label{familyp}
Let $H \ge 5$. Assuming Dickson's conjecture, one can find $H$ \textit{Sophie Germain primes} 
$p_1, p_2, \ldots, p_H$  such that 
\begin{equation}\label{orderp}
 p_1 > 4H+1\;  \text{ and for any $h$ in } [1, H-1]\; \colon \; p_{h+1} > 3p_h(2p_h+1)^3.
\end{equation}
Further, for all $h, k \in [1, H]$ with $h \neq k$ and for all $p \in [5, H]$, these
\textit{Sophie Germain primes} satisfy the following relations;
\begin{eqnarray}
\label{modpiH1}
\gcd(2 h + p_h, ~p) &=& 1, \phantom{mm}\\
\label{modpiH2}
\gcd\left(2h + p_h(2p_h+1)^3, ~p\right) &=&1, \phantom{mm}\\
\label{deux}
 \gcd\left(p_h -2(k-h), ~p_k(2p_k+1)\right) &=&1, \phantom{mm}\\
\label{trois}
\text{ and   } 
\gcd\left( p_h(2p_h + 1)^3- 2(k-h), ~p_k(2p_k+1) \right)&=&1. \phantom{mm}
\end{eqnarray}
\end{prop}

\noindent \textit{Proof of Proposition \ref{familyp}}\\

We let
$$
\Pi_H = \prod_{\substack{5 \le p \le H\\ p\text{ prime}}} p.
$$

We prove by induction that for any $h$ between $1$ and $H$, we can 
find Germain primes $p_1, p_2, \ldots, p_h$ such that
\begin{equation}
\label{orderpind}
p_1 > 4H+1 \;  \text{ and for any 
$\ell$ in } [1, h-1]\; \colon \; p_{\ell+1} > 3p_{\ell}(2p_{\ell}+1)^3 .
\end{equation}
Further, 
\begin{eqnarray}
\label{modpiHind1}
\forall \ell \in [1, h], \forall p \in [5, H] \colon \phantom{mmmm}
\gcd \left(2 \ell + p_{\ell}, ~p \right) &=& 1, \phantom{mm} \\
\label{modpiHind2}
\text{ and} \phantom{m}
 \gcd\left(2\ell + p_{\ell}(2p_{\ell}+1)^3, ~p\right)&=&1, \phantom{mm}\\
\label{klel1}
\forall k <  \ell \le h  \colon  \phantom{m}\!\! 
\gcd(p_{\ell} -2(k-\ell), ~p_k(2p_k+1)) &=&1, \phantom{mm}\\
\label{klel2}
\text{ and }\phantom{m}
\gcd( p_{\ell}(2p_{\ell}+1)^3- 2(k-\ell), ~p_k(2p_k+1))&=&1, \phantom{mm}\\
\label{llek1}
\forall k <  \ell \le h  \colon  \phantom{m}
\gcd(p_{k} -2(\ell -k), ~p_{\ell}(2p_{\ell}+1)) &=&1, \phantom{mm}\\
 \label{llek2}
\text{  and } \phantom{m}
 \gcd( p_{k}(2p_k+1)^3- 2(\ell - k), ~p_{\ell}(2p_{\ell}+1))&=&1. \phantom{mm}
\end{eqnarray}

\indent Construction of $p_{1}$. 

\smallskip
For $h =1$, conditions (\ref{klel1}) to (\ref{llek2}) are empty. 
Condition (\ref{orderpind}) will be
satisfied as soon as we know that there are infinitely many 
Germain primes satisfying (\ref{modpiHind1}) and (\ref{modpiHind2}).

For $p \in [11,  H]$, we can always find a residue class 
$r_1(p)$ modulo $p$ such that none of
$r_1(p)$, $2r_1(p) +1$, $2 + r_1(p)$ and 
$2  + r_1(p) (2r_1(p) + 1)^3$ are 
equivalent to $0$ modulo $p$. This is possible as we have to avoid at most  
$7 \,(=1+1+1+4)$ residue classes modulo $p$.  \\

This is also true for $p =5$ and $p=7$: indeed, for $p=5$, 
the set of residue classes $A_5 := \{1(5),~ 1\times (2 \times 1 +1)^3(5) \}$ is 
disjoint from the set of residue classes $B_5 := \{ 3(5), ~ 3\times (2 \times 3 +1)^3(5)\}$ 
and that for $p=7$, the set of residue classes 
$A_7 := \{ 1(7), ~1\times (2 \times 1 +1)^3(7) \}$ 
is disjoint from the set of residue classes 
$B_7 := \{2(7), ~2\times (2 \times 2 +1)^3(7) \}$.
Hence at least one of the sets among $A_5 + 2$ and $B_5 + 2$ 
consists of non-zero residue classes. Same applies to at least one of the
sets among $A_7 + 2$ and  $B_7 + 2$. \\

Having found suitable residue classes $r_1(p)$ for any prime 
$p \in [5, H]$, the Chinese remainder theorem permits us to find a positive 
integer $s(1)$ such that, 
for each prime $p \in [5, H]$, none of the numbers $s(1)$, $2s(1) +1$, 
$2+s(1)$ and $2+s(1)(2s(1)+1)^3$ is congruent to 
$0$ modulo $p$. Thus, by Lemma \ref{ginap} the arithmetic progression 
with difference $\Pi_H$ and first term $s(1)$ contains infinitely many 
Germain primes satisfying (\ref{modpiHind1}) and (\ref{modpiHind2}), 
and we can 
thus find such a prime satisfying also (\ref{orderpind}).\\

\indent We now apply induction to complete the proof of Proposition 1. 

Let us now assume that for some $h$ between $1$ and $H-1$ we have 
constructed a family of $h$ Germain primes satisfying relations 
(\ref{orderpind}) to (\ref{llek2}). Now we would like to construct $p_{h+1}$. It is enough to 
show that there exist infinitely many Germain primes $p_{\ell}$ satisfying 
the relations (\ref{modpiHind1}) to (\ref{llek2}) where $\ell$ and $h$ are 
replaced by $h+1$. Our new relations (\ref{llek1}) and (\ref{llek2}) are trivially 
satisfied as soon as $p_{h+1}$ is large enough. For each $\ell < h+1$, one can 
choose an integer $r_{h+1}(\ell)$  satisfying the following relations
\begin{eqnarray*}
\gcd \left(r_{h+1} (\ell),  ~p_{\ell} (2p_{\ell} + 1) \right) &=& 1, \\
\gcd \left(2r_{h+1}(\ell)+1, ~p_{\ell} (2p_{\ell}+1)\right) &=&1,\\
\gcd \left(r_{h+1}(\ell) -2(\ell - (h+1)),  ~p_{\ell}(2p_{\ell} + 1)\right) &=&1,\\
\gcd \left( r_{h+1}(\ell)(2r_{h+1}({\ell})+1)^3- 2(\ell - (h+1)),  
~p_{\ell}(2p_{\ell}+1)\right)&=&1. 
\end{eqnarray*}
This is possible as we need to avoid at most 14 residue classes modulo
$p_{\ell}(2p_{\ell}+1)$. 
Arguing as we did previously, we can find a positive integer $s(h+1)$ such that 
all the numbers $s(h+1)$, $2s(h+1) +1$, $2(h+1)+s(h+1)$ and 
$2 (h+1)+s(h+1)(2s(h+1)+1)^3$ are coprime to $\Pi_H$.\\
By the Chinese remainder theorem and Dickson's conjecture, there exist infinitely many 
Germain primes which satisfy the relations (\ref{modpiHind1}) to (\ref{klel2}), 
and we can choose one of them which is sufficiently large to satisfy 
also (\ref{orderpind}), (\ref{llek1}) and (\ref{llek2}); we call such a prime $p_{h+1}$.

\smallskip
This completes the proof of Proposition \ref{familyp}.
\qed

\bigskip

\textbf{Step 3}. Construction of an auxiliary polynomial $F$.\\

We consider the set $\{p_1, p_2, \dots, p_H\}$ introduced in
Proposition \ref{familyp} and for, $h \in [1, H]$, we define the integer $n_h$ by
\begin{equation}\label{defnh}
n_h = \begin{cases}
		2 p_h, & \text{ if } h \not\equiv 2 \pmod 3,\\
		2p_h (2p_h+1)^3 & \text{ if } h \equiv 2 \pmod 3.
		\end{cases}
\end{equation}
We notice that, thanks to relation (\ref{orderp}), the numbers $n_h/2$
as $h$ varies from $1$ to $H$ are pairwise coprime. We let
$$
U= 12 \Pi_H \prod_{h=1}^H n_h^5 
$$
and we select a positive integer $r$ such that
\begin{equation}\label{congr}
r \equiv \begin{cases}
		1 & \pmod{ 12 \Pi_H}\\
		1 - 4h & \pmod{ (n_h/2)^2}\; \text{ for all integer } h \in [1, H].
		\end{cases}
\end{equation}
For $h \in [1, H]$, we define the polynomial $F_h$ by
\begin{equation}\label{Fh}
F_h(t) = \frac{Ut+r-1+4h}{n_{h}} +1 
\end{equation}
and we let
$$
F = \prod_{h=1}^H F_h.
$$
It is readily seen that each $F_h$ is a linear polynomial with integer coefficients
and positive leading coefficient.\\

\textbf{Step 4}. The polynomial $F$ has no fixed prime divisor.\\

If $p$ does not divide $U$,  the congruence 
$F(t) \equiv 0 \pmod p$ 
has at most $H$ solutions in $\mathbb{Z}/p\mathbb{Z}$; 
but $p$ is larger than $H$ 
and thus $p$ is not a fixed divisor of $F$.

If $p$ divides $U$, then either $p \in [2, H]$ or $p= p_h, 2p_h +1$
for some $1 \le h \le H$. In this case, the relation
\begin{equation}\label{fix}
F(t) \equiv 0 \!\!\! \pmod p \iff
\prod_{h=1}^H \frac{r-1+4h}{n_h} + 1 \equiv 0 \!\! \! \pmod p. 
\end{equation}
Note that for any $h \in [1, H]$, we have
$r-1 + 4h \equiv 0 \!\! \pmod {n_h^2}$
and hence $p=2$ can never satisfy \eqref{fix} and also
$F_h(t) \equiv 1 \!\! \pmod {n_h}$. Now suppose that $p = p_h$ or $2p_h+1$  
divides $(r- 1 + 4k)/n_k + 1$ for some $k \ne h$. Then 
we have
$$ 
2(k-h) + \frac{n_k}{2} \equiv 0 \!\!\! \pmod p,
$$
a contradiction to the relations \eqref{deux} and \eqref{trois}.

Finally, if any prime $p \in [3, H]$ satisfies \eqref{fix}, then we have
$$
2h + \frac{n_h}{2} \equiv 0 \!\!\! \pmod {p},
$$
a contradiction to relations \eqref{modpiH1}, \eqref{modpiH2} 
and the fact that any Germain prime $p_G > 3$ is always 
congruent to $2$ modulo $3$ and that
$p_G(2p_G+1)^3$ is congruent to $1$ modulo $3$.

\smallskip
Thus, the polynomial $F$ has no fixed divisor.\\

\textbf{Step 5}. Finding integers satisfying (\ref{estrong}).\\

Since the polynomials $F_h$ have positive leading coefficients
and $F$ has no fixed prime divisor, 
Dickson's conjecture implies that one can find a positive integer $t_0$ such that for 
each $h \in[1, H]$ the value of the polynomials 
$F_h$ at $t_0$ are prime numbers. Let us write $q_h = F_h(t_0)$ and $M=Ut_0+r-1$.

We thus have
\begin{equation}\label{qnu}
q_h= \frac{M+4h}{n_h} +1.
\end{equation}
We notice that $n_h$ is a value of the $\varphi$ function, since $\varphi(2p_h+1)= 2p_h$ 
and $\varphi((2p_h+1)^4)=2p_h(2p_h+1)^3$. Let us write $\nu_h$ to be 
$2p_h+1$ or $(2p_h+1)^4$ so that $n_h = \varphi(\nu_h)$. We have for all $h$
\begin{equation}\notag
M+4h = \varphi(q_h)\varphi(\nu_h) = \varphi(q_h\nu_h),
\end{equation}
the last equality coming from the fact that, by construction, $q_h$ is a prime 
larger than $\nu_h$ and thus coprime with it.

Theorem \ref{strongmain} is thus proved.
\qed\\

\noindent \textbf{Remark 2.}\;  
The reader will easily notice that with our construction, the sequence $(m_h)_h$ is 
decreasing and that a similar proof would permit to show that for any permutation 
$\tau$ of the set $\{1, 2, \ldots, H\}$, one can find a sequence of integers $(m_h)_h$ 
satisfying Theorem \ref{strongmain} such that the sequence $(m_{\tau(h)})_h$ is decreasing.\\

\noindent \textbf{Remark 3.}\;
The use of Dickson's conjecture is not necessary in the construction of a suitable 
polynomial $F$. By using a sieve method, one can construct a polynomial 
$$
F(t) = \prod_{1 \le h \le H} \left(\frac{Ut+r-1+ 4h}{u_h}+1\right),
$$
such that each of its factors has integral coefficients, which has no fixed prime 
divisor and such that for each $h \in [1, H]$ the number $u_h+1$ is prime.\\

\section{Proof of Theorem \ref{main}}\label{three}

As mentioned in the introduction, Theorem \ref{strongmain} easily implies that the Banach density of the set $\mathcal{V}$ 
is at least $1/4$.\\

Let us show why the opposite inequality is a direct consequence of the following result of K. Ford, S. Konyagin 
and C. Pomerance.

\begin{thm}[Theorem 2 of \cite{FKP}]\label{fkp}
For any $\varepsilon > 0$ there exist such $m$ that at least $(1-\varepsilon)m$ residue classes $a (\operatorname{mod} 4m), 0 < a < 4m, a\equiv 2 (\operatorname{mod} 4)$ are totient free. 
\end{thm}

Since $\varphi(n)$ is even except when $n$ is equal to $1$ or $2$, Theorem \ref{fkp} implies that for any positive real number $\varepsilon$ less than $1$, we can find an $m$ such that at least $(3-\varepsilon)m$ residue classes $a (\operatorname{mod} 4m), 0 < a < 4m$ are totient free. For any  positive $H$ and $x$ we have 
\begin{eqnarray}\notag
\Card\{\mathcal{V}\cap(x, x+H]\} &\le& \left(\left\lfloor\frac{H}{4m}\right\rfloor +1 \right) (1 + \varepsilon) m\\
\notag
&\le& (1 + \varepsilon) \frac{H}{4} + 2m.
\end{eqnarray}
This last inequality and the definition of the Banach density (\ref{defdelta}) imply $\delta^*(\mathcal{V}) \le 1/4$.

\section{Acknowledgements}
The authors are thankful to Kevin Ford who drew their attention to the paper \cite{FKP}.\\
This work benefitted from the support of the SPARC project 445. The first 
author is thankful to the Institute of Mathematical Sciences, Chennai, 
and the Harish-Chandra Research Institute, Prayagraj, for 
providing excellent conditions for working. The first and the third author would
also like to thank the Indo-French program (IFPM) in mathematics. 
The third author would also like to thank MTR/2018/000201 and DAE number theory 
plan project for partial financial support.

\end{document}